\documentclass[10pt]{article}
\usepackage{amssymb,latexsym,amsmath}
\usepackage{amsthm}
\usepackage{geometry}
\usepackage{lineno}
\usepackage[T1]{fontenc}
\usepackage[utf8]{inputenc}
\usepackage{authblk}
\newtheorem{thm}{Theorem}[subsection]

\newtheorem{lem}[thm]{Lemma}

\theoremstyle{plain}

\theoremstyle{definition}

\newtheorem{example}{Example}[section]
\numberwithin {equation}{section}
\usepackage{longtable}

\title
{A new  effective  weighted modified perturbation technique for solving a class of hypersingular integral equations
}
\author[1]{Mostafa Akrami\thanks{Corresponding author:mostafa.akrami@mun.ca}}
\author[2]{Taher Lotfi\thanks{lotfitaher@yahoo.com, lotfi@iauh.ac.ir}}
\author[2]{Farajollah Mohammadi Yaghoobi\thanks{fm.yaghoobi@gmail.com}}

\affil[1]{Department of Earth Sciences, Memorial University of Newfoundland, St. John's, NL, A1B 3X7, Canada.}
\affil[2]{Department of Mathematics, Hamedan Branch, Islamic Azad
University, Hamedan 65138, Iran.}

\begin{document}

\date{}
\maketitle



\begin{abstract}
\noindent This paper is an attempt to
solve an important class of hypersingular integral equations of the second kind. To this end, we apply a new weighted and modified perturbation method  which includes some special cases of the  Adomian decomposition method. To justify the efficiency   and
applicability of the proposed method,  we examine some examples. The
principal aspects of this method are its simplicity along with fast
computations.

\end{abstract}
\section{Introduction}
Enormous  physical and engineering problems, such as crack problems in fracture mechanics \cite{Chan}, water wave scattering problems involving barriers \cite{Kaya}, diffraction of electromagnetic wave and aerodynamic problems \cite{Lifanov,Mandal1},   generalization of the elliptic wing case of Prandtl's equation \cite{Mandal2,Mandal3}  are  hypersingular integral  equations of the second kind. For more applications one can consult \cite{Dutta, Gori,Chen, Chakrabarti,Wazwaz}.

In this work, we consider the following hypersingular integral equation of the second kind
\begin{equation}\label{s1-f1}
u(x)=\sqrt{1-x^2}f(x)+\frac{\alpha}{\pi}\sqrt{1-x^2}\int_{-1}^{1}\frac{u(t)}{(t-x)^n}dt,\qquad n=2,3,\dots,\,\,-1<x<1,
\end{equation}
where $\alpha$ is a real positive number, $f(x)$ is a given
function, and $u(x)$ is unknown. Moreover, $u(\pm 1)=0$. Under
these assumptions, Eq. (\ref{s1-f1}) can be called the generalized
case for the oval wing of Prandtl's equation. In addition, Eq. (\ref{s1-f1}) is
referred to as Hadamard finite part \cite{Mandal1}
\begin{equation}\label{s1-f2}
\int_{-1}^{1}\frac{u(t)}{(t-x)^2}dt=\lim_{\epsilon\rightarrow0^+}\bigg[\int_{-1}^{x-\epsilon}\frac{u(t)}{(t-x)^2}dt
+\int_{x+\epsilon}^{1}\frac{u(t)}{(t-x)^2}dt-\frac{u(x+\epsilon)+u(x-\epsilon)}{\epsilon}\bigg].
\end{equation}
Recently, Mahmoudi  introduced an efficient method to solve (\ref{s1-f1}) based on the weighted modified Adomian's decomposition method (WMADM). Motivated by his work, we attempt  to  solve Eq. (\ref{s1-f1}) by applying the weighted modified  perturbation method (WMPM),  which includes the mentioned method as the special case. The rest of this paper is organized as follows: Section 2 is devoted to the description of the solution method. It contains the classic  perturbation method (PM) which does not work even for simple examples here. Therefore, we try to put forward the modified  and  weighted modified version of the classic PM in such a way that it is possible to find a solution quickly. In other words, the introduced weighted modified version can save many of the computations and prevents the construction of commonly used series solutions which can be considered other  virtues for these methods. To support the given method, some applications and illustrations are illustrated  in Section 3 . Finally, Section 4 includes some  concluding remarks.
\section{Description of the solution method}
This section deals with applying the perturbation technique for
solving hypersingular integral equation of the second kind
(\ref{s1-f1}). Before describing the whole process, for the sake
of simplicity in computations \cite{Mandal1}, in  Eq.
(\ref{s1-f1}) it is assumed that:
\begin{equation}\label{s2-f51}
u(x)=\sqrt{1-x^2}\psi(x),
\end{equation}
where $\psi(x)$ is a smooth function. Substituting (\ref{s2-f51})
into (\ref{s1-f1}) and simplifying, it induces:
\begin{equation}\label{s2-f5}
\psi(x)=f(x)+\frac{\alpha}{\pi
\Gamma(n)}\frac{d^{n-1}}{dx^{n-1}}\int_{-1}^{1}\frac{\sqrt{1-t^2}\psi(t)}{t-x}dt.
\end{equation}
We need the following  lemma \cite{Mandal2}:
\begin{lem}\label{lem} If
\[I_j(x)=\int_{-1}^{1}\frac{\sqrt{1-t^2}}{t-x}\,t^j\,dt,\quad j=0,1,\cdots,\]
then
\begin{align}\label{s2-f6}
I_j(x)=-\pi
x^{j+1}+\sum_{i=0}^{j-1}\frac{1+(-1)^i}{4}\frac{\Gamma(1/2)\Gamma((i+1)/2)}{\Gamma((i+4)/2)}x^{j-i-1}\ \ \
j=0,1,\cdots,
\end{align}
where $\Gamma(n)$ stands for the Gamma function.
\end{lem}
Also, we have the following lemma:
\begin{lem}\label{lem}
For $k\geq 1$,
\begin{equation}
I_{2k}(x)=xI_{2k-1}(x).
\end{equation}
\end{lem}
\begin{proof} By (\ref{s2-f6}), we have:
\begin{equation}\begin{split}
I_{2k}(x)&=-\pi
x^{2k+1}+\sum_{i=0}^{2k-1}\frac{1+(-1)^i}{4}\frac{\Gamma(1/2)\Gamma((i+1)/2)}{\Gamma((i+4)/2)}x^{2k-i-1}\\&=
x\big(-\pi
x^{2k}+\sum_{i=0}^{2k-2}\frac{1+(-1)^i}{4}\frac{\Gamma(1/2)\Gamma((i+1)/2)}{\Gamma((i+4)/2)}x^{2k-i-2}\big)=xI_{2k-1}(x)
\end{split}\end{equation}
\end{proof}
Here, we try to explain the new method for solving (\ref{s1-f1}).  For this purpose, based upon the Eq. (\ref{s2-f5}), we define the operator:
\begin{equation}\label{s2-f1}
L(\psi)=L_1(\psi)+L_2(\psi)=f(x).
\end{equation}
Next, we consider the perturbation technique as follows:
\begin{equation}\label{s2-f2}
L_1(\psi)=f(x)+\varepsilon L_2(\psi).
\end{equation}
 Now,  the stage is ready for computing with the perturbation
technique. Considering Eq. (\ref{s2-f5}), we set:
\begin{equation}\label{s2-f7}
L_1(\psi)=\psi(x), \ \ \ L_2(\psi)=-\frac{\alpha}{\pi
\Gamma(n)}\frac{d^{n-1}}{dx^{n-1}}\int_{-1}^{1}\frac{\sqrt{1-t^2}\psi(t)}{t-x}dt,
\end{equation}
where
\begin{equation}\label{s2-f8}
\psi(x)=\lim_{\varepsilon\to
1}\sum_{k=0}^{\infty}\psi_k(x)\varepsilon^k.
\end{equation}
The series (\ref{s2-f8}) converges to the exact solution of Eq.
(\ref{s2-f5}) provided that such a solution exists \cite{Bellman,Nayfeh}. Substituting
(\ref{s2-f7}) and (\ref{s2-f8}) into (\ref{s2-f2}), one obtains
\begin{equation}\label{s2-f9}
\sum_{k=0}^{\infty}\psi_k(x)\varepsilon^k=f(x)+\frac{\alpha}{\pi
\Gamma(n)}\frac{d^{n-1}}{dx^{n-1}}\sum_{k=0}^{\infty}\varepsilon^{k+1}\int_{-1}^{1}\frac{\sqrt{1-t^2}\psi_k(t)}{t-x}dt.
\end{equation}
Finally, equating the terms with like powers of the embedding
parameter $\varepsilon$, we can compute $\psi_n$ by the following
recurrence relation:

\begin{equation}\label{s2-f10}
\begin{cases}
\varepsilon^0:&\ \psi_0(x)=f(x),\\
\varepsilon^k:&\ \psi_{k}(x)=\frac{\alpha}{\pi
\Gamma(n)}\frac{d^{n-1}}{dx^{n-1}}\int_{-1}^{1}\frac{\sqrt{1-t^2}}{t-x}\,\psi_{k-1}(t)\,dt,
\ \ k=1,2,\cdots.
\end{cases}
\end{equation}
Note that this method generates the Adomian decomposition method (ADM), and  we  state this fact formally in the following theorem.
\begin{thm}\label{Theorem1} If $\psi_0(x)=f(x)$, then the relation
\[\psi_k(x)=
\frac{\alpha}{\pi
\Gamma(n)}\frac{d^{n-1}}{dx^{n-1}}\int_{-1}^{1}\frac{\sqrt{1-t^2}}{t-x}\,\psi_{k-1}(t)\,dt,\quad k=1,2,\dots,\]
generates Adomian decomposition method for solving (\ref{s1-f1}).
\end{thm}
\subsection{Modified perturbation method}
In some situations, Formula (\ref{s2-f10}) does not work perfectly
or the solution series converges slowly. To overcome these
difficulties, it is still possible to modify the proposed
perturbation technique (\ref{s2-f10}) as follows. First, we split
$f(x)$ as $f(x)=f_1(x)+f_2(x)$. Then, we set
\begin{equation}\label{s2-f11}
L_1(\psi)=\psi, \ \ \ L_2(\psi)=-\frac{\alpha}{\pi
\Gamma(n)}\frac{d^{n-1}}{dx^{n-1}}\int_{-1}^{1}\frac{\sqrt{1-t^2}\psi}{t-x}dt+f_2(x).
\end{equation}
At once, we can modify (\ref{s2-f10}) and call it modified
perturbation technique given by:
\begin{equation}\label{s2-f12}
\begin{cases}
\psi_0(x)=f_1(x),\\
\psi_1(x)=f_2(x)+\frac{\alpha}{\pi \Gamma(n)}\frac{d^{n-1}}{dx^{n-1}}\int_{-1}^{1}\frac{\sqrt{1-t^2}}{t-x}\,\psi_0(t)\,dt,\\
\psi_{k}(x)=\frac{\alpha}{\pi
\Gamma(n)}\frac{d^{n-1}}{dx^{n-1}}\int_{-1}^{1}\frac{\sqrt{1-t^2}}{t-x}\,\psi_{k-1}(t)\,dt,
\ \ k=2,3,\cdots.
\end{cases}
\end{equation}
Similar to the Theorem (\ref{Theorem1}), we have:

\begin{thm}\label{Theorem2} If $\psi_0(x)=f_1(x)$, $\psi_1(x)=f_2(x)+\frac{\alpha}{\pi \Gamma(n)}\frac{d^{n-1}}{dx^{n-1}}\int_{-1}^{1}\frac{\sqrt{1-t^2}}{t-x}\,\psi_0(t)\,dt$, then
\[\psi_{n+1}(x)=\frac{\alpha}{\pi \Gamma(n)}\frac{d^{n-1}}{dx^{n-1}}\int_{-1}^{1}\frac{\sqrt{1-t^2}}{t-x}\,\psi_n(t)\,dt,\quad n=1,2\dots,\]
 generates the  modified
Adomian decomposition method (MADM) for solving (\ref{s1-f1}).
\end{thm}

\subsection{Weighted modified perturbation method}
Choosing functions $f_1(x)$ and $f_2(x)$ in applying the modified
perturbation technique deeply affects the rate of convergence.
Generally speaking, they should be chosen so that $\psi_k=0$ for
$k=1,2,\cdots$. To select these functions effectively, the weighted
modified perturbation technique is advised. So, we  pay attention
to this issue when $f(x)$ is a polynomial and leave  the other
cases for future research. More details can be found in
\cite{Abbasbandy1,Abbasbandy2,Mahmoudi,Wazwaz}.

Suppose that $f(x)$ in (\ref{s2-f5}) is a polynomial with degree
$m$. Hence, $\psi(x)$ must be a polynomial of the same degree. In
other words, let:
\begin{equation}\label{s2-f13}
f(x)=\sum_{i=0}^ma_i\,x^i.
\end{equation}
It is possible to write $f(x)=f_1(x)+f_2(x)$, where
\begin{align}\label{s2-f14}
f_1(x)&=\sum_{i=0}^mb_i\,x^i,\notag\\
f_2(x)&=\sum_{i=0}^m(a_i-b_i)\,x^i.
\end{align}
The unknowns $b_i$, $i=0,1,\cdots,m$ are determined in the modified perturbation technique
(\ref{s2-f12})in such a way  that
$\psi_1(x)=0$ . This process is called the weighted modified
perturbation technique and will be illustrated in the next section.\\

\section{Applications and illustrations}
This section is concerned with applying the modified perturbation
technique (\ref{s2-f10}) or the weighted modified perturbation
technique (\ref{s2-f12}) to solve some concretes of the hypersingular integral
equation (\ref{s1-f1}) or equivalents (\ref{s2-f5}).

\begin{example} \cite{Mahmoudi} Consider:

\begin{equation}\label{s3-f1}
\begin{split}
u(x)=2\pi\sqrt{1-x^2}+\frac{1}{2}\sqrt{1-x^2}\int_{-1}^{1}\frac{u(t)}{(t-x)^2}dt,
\ \ \ -1<x<1.
\end{split}\end{equation}

Substituting    $u(x)=\sqrt{1-x^2}\psi(x)$ into (\ref{s3-f1}), one obtains

\begin{equation}\label{s3-f2}
\begin{split}
\psi(x)-\frac{1}{2}\frac{d}{dx}\int_{-1}^{1}\frac{\sqrt{1-t^2}\psi(t)}{t-x}dt=2\pi,
\ \ \ -1<x<1.
\end{split}\end{equation}

By  perturbation technique (\ref{s2-f10}), Lemma
(\ref{lem}), and  (\ref{s3-f2}), we have:
\begin{equation}\begin{split}
\psi_n(x)=\frac{(-1)^n \pi^{n+1}}{2^{n-1}}, \ \ n=0,1,2,\cdots.
\end{split}\end{equation}

Therefore, from (\ref{s2-f8}) we have:
\begin{equation}\begin{split}
&\psi(x)=\sum_{n=0}^{\infty}\frac{(-1)^n \pi^{n+1}}{2^{n-1}},
\end{split}\end{equation}
which is a divergent series since $\lim_{n\to\infty}{\psi_n}\neq
0$.

So, we  apply  the weighted modified
perturbation technique (\ref{s2-f12}). Let  $f_1(x)=b_0$ and
$f_2(x)=2\pi-b_0$. Thus
\begin{equation}\begin{split}
&\psi_0(x)=b_0, \ \ \  \psi_1(x)=2\pi-b_0-b_0\frac{\pi}{2}.
\end{split}\end{equation}
Putting $\psi_1(x)=0$ leads to:

\begin{equation*}\begin{split}
b_0=\frac{4\pi}{\pi+2}.
\end{split}\end{equation*}
Consequently,
\begin{equation}\begin{split}
\psi(x)=\frac{4\pi}{\pi+2},
\end{split}\end{equation}
and

\begin{equation}\begin{split}
u(x)=\frac{4\pi}{\pi+2}\sqrt{1-x^2},
\end{split}\end{equation}
which are  the exact solutions of the equivalent  hypersingular integral equations (\ref{s3-f2}) and  (\ref{s3-f1}), respectively.

\end{example}

\begin{example} In this example, we try to solve the general case of the hypersingular integral equation by using the weighted modified perturbation technique. Let:

\begin{equation}\label{s3-f18}
\begin{split}
u(x)=\big(a_0+a_1x+a_2x^2+\cdots+
a_mx^m\big)\sqrt{1-x^2}+\frac{\alpha}{\pi}\sqrt{1-x^2}\int_{-1}^{1}\frac{u(t)}{(t-x)^n}dt,
\ \ n\geq2,\ \ \ -1<x<1.
\end{split}\end{equation}

Substituting  $u(x)=\sqrt{1-x^2}\psi(x)$ into (\ref{s3-f18})
causes

\begin{equation}\label{s3-f19}
\begin{split}
\psi(x)=a_0+a_1x+a_2x^2+a_3x^3+\cdots+ a_mx^m+\frac{\alpha}{\pi
\Gamma(n)}\frac{d^{n-1}}{dx^{n-1}}\int_{-1}^{1}\frac{\sqrt{1-t^2}\psi(t)}{(t-x)}dt=,
\ \ \ -1<x<1.
\end{split}\end{equation}
Set
\begin{equation*}\begin{split}
&f_1(x)=b_0+b_1x+b_2x^2+b_3x^3+\cdots+b_mx^m,\\
&f_2(x)=(a_0-b_0)+(a_1-b_1)x+(a_2-b_2)x^2+(a_3-b_3)x^3+\cdots+(a_m-b_m)x^m.
\end{split}\end{equation*}

Therefore, to obtain $\psi_1(x)=0$ , by Lemma  (\ref{lem}) and the
weighted modified perturbation technique  (\ref{s2-f12}), we
acquire

\begin{equation*}
\left\{\begin{array}{cc}\begin{split}
 &b_s=a_s, \ \ \ m-n+3\leq s \leq m,\\
&b_{m-n+2}=a_{m-n+2}-\frac{\alpha}{\Gamma(n)}(m+1)m\cdots (m-n+3)b_m,\\
&\vdots\\
&b_{m-2k-n+2}=a_{m-2k-n+2}-\frac{\alpha}{\Gamma(n)}(m-2k+1)(m-2k)\cdots(m-2k-n+3)b_{m-2k}\\
&\qquad\qquad\quad+\frac{\alpha}{\Gamma(n)}(m-2k+1)(m-2k)\cdots(m-2k-n+3)\sum_{i=1}^{k}{\frac{\Gamma(2i-1)}{2^{2i-1}\Gamma(i+1)\Gamma(i)}b_{(m-2k+2i)}},\\
&b_{m-2k-n+1}=a_{m-2k-n+1}-\frac{\alpha}{\Gamma(n)}(m-2k)(m-2k-1)\cdots(m-2k-n+2)b_{m-2k-1}\\
&\qquad\qquad\quad+\frac{\alpha}{\Gamma(n)}(m-2k)(m-2k-1)\cdots(m-2k-n+2)\sum_{i=1}^{k}{\frac{\Gamma(2i-1)}{2^{2i-1}\Gamma(i+1)\Gamma(i)}b_{(m-1-2k+2i)}}.
\end{split}\end{array}\right.
 \end{equation*}
Consequently,   the exact solution of the integral equation
(\ref{s3-f18}) is given by:
\begin{equation*}\begin{split}
u(x)=\psi(x)\sqrt{1-x^2}=\psi_0(x)\sqrt{1-x^2}=(b_0+b_1x+b_2x^2+b_3x^3+\cdots+b_mx^m)\sqrt{1-x^2}.
\end{split}\end{equation*}
\end{example}

\section{Concluding remarks}
This paper has introduced a new method for solving an important class of hypersingular integral equations of the second kind. The main features of the proposed method are avoiding complicated procedures and having simplicity. Using  an example, we have shown that the classic perturbation method does not work while its weighted modified version can remedy this issue. 
\section{Acknowledgments}
The  authors appreciate Hamedan Branch of Islamic Azad University for the financial support of this research.

\end{document}